\documentclass{jloganal}
\usepackage{graphicx,tikz} 
\usetikzlibrary{matrix,arrows,decorations.pathmorphing}
\usepackage[all,cmtip]{xy}

\newtheorem{theorem}{Theorem}[section]
\newtheorem{prop}[theorem]{Proposition}
\newtheorem{lemma}[theorem]{Lemma}
\newtheorem{cor}[theorem]{Corollary}
\theoremstyle{definition}
\newtheorem{definition}[theorem]{Definition}
\newtheorem{remark}[theorem]{Remark}
\newtheorem{example}[theorem]{Example}

\newcommand{\rr}{\mathbb{R}}
\newcommand{\qq}{\mathbb{Q}}

\newcommand{\cc}{\mathbb{C}}
\newcommand{\nn}{\mathbb{N}}
\newcommand{\zz}{\mathbb{Z}}

\newcommand{\nsr}{{^*\mathbb{R}}}

\newcommand{\nsn}{{^*\mathbb{N}}}
\newcommand{\nsz}{{^*\mathbb{Z}}}

\newcommand{\noo}{{^*\mathcal{O}}}

\newcommand{\oo}{\mathcal{O}}
\newcommand{\mm}{\mathfrak{m}}

\newcommand{\aaa}{\mathfrak{a}}

\newcommand{\bbb}{\mathfrak{b}}
\newcommand{\ccc}{\mathfrak{c}}

\newcommand{\pp}{\mathfrak{p}}

\newcommand{\srad}{{^*Rad}\:}
\newcommand{\rad}{Rad\:}

\newcommand{\fin}{\text{fin}\hspace{0.5mm} }

\renewcommand{\theenumi}{\alph{enumi}}

\DeclareFontFamily{OT1}{pzc}{}
\DeclareFontShape{OT1}{pzc}{m}{it}{<-> s * [1.100] pzcmi7t}{}
\DeclareMathAlphabet{\mathpzc}{OT1}{pzc}{m}{it}

\title[Nonstandard Dedekind Rings]{On the Spectrum of Nonstandard Dedekind Rings }
\author[H Knospe]{Heiko Knospe}
\givenname{Heiko}
\surname{Knospe}
\address{Technische Hochschule K\"oln,
Institut f\"ur Nachrichtentechnik,
Betzdorfer Str. 2,
50679 K\"oln, Germany}
\email{heiko.knospe@th-koeln.de}
\urladdr{}
\author[C Serp\'{e}]{Christian Serp\'{e}}
\givenname{Christian}
\surname{Serp\'{e}}
\address{Westf\"alische Wilhelms-Universit\"at M\"unster,
Fachbereich Mathematik und Informatik,
Einstein\-str. 62,
48149 M\"unster,
Germany}
\email{serpe@uni-muenster.de}
\urladdr{}

\arxivreference{math.NT/1509.06203}
\volumenumber{}
\issuenumber{}
\publicationyear{}
\papernumber{}
\startpage{}
\endpage{}
\doi{}
\MR{}
\Zbl{}
\received{}
\revised{}
\accepted{}
\published{}
\publishedonline{}
\proposed{}
\seconded{}
\corresponding{}
\editor{}
\version{}

\begin{document}


\begin{abstract}
The methods of nonstandard analysis are applied to algebra and number theory. We study nonstandard Dedekind rings, for example an ultraproduct of the ring of integers of a number field. Such rings possess a rich structure and have interesting relations to standard Dedekind rings and their completions.  We use lattice theory to classify the ideals of nonstandard Dedekind rings. The nonzero prime ideals are contained in exactly one maximal ideal and can be described using valuation theory. The localisation at a prime ideal gives a valuation ring and we determine the value group and the residue field. The spectrum of a nonstandard Dedekind ring is described using lattices and value groups.
Furthermore, we investigate the Riemann-Zariski space and the valuation spectrum of nonstandard Dedekind rings and their quotient fields. 
\end{abstract}

\begin{asciiabstract}

\end{asciiabstract}

\maketitle



\section{Introduction}

 Nonstandard analysis was invented in the 1960s by Abraham Robinson (\cite{robinson1961}, \cite{robinson1996}) and nonstandard methods have been successfully applied to calculus, functional analysis and stochastic in the last decades. Algebraic and arithmetic applications also exist but are less common. The fundamental contributions go back to A. Robinson (\cite{algnum}, \cite{dedekind}, \cite{arithmetic}, \cite{topics}) and a major result was the nonstandard proof of the Siegel-Mahler Theorem by A. Robinson and P. Roquette (\cite{robinson-roquette}).  It was shown that there is a relation between the geometry of varieties over number fields  and the number theory of nonstandard number fields (\cite{macintyre}, \cite{kani1980}). Ideals in nonstandard number fields were studied in \cite{cherlin}, \cite{cherlindedekind} and \cite{klingen1975}.

The main aim of this paper is to recall,  reformulate and extend the results on the ideal structure of nonstandard Dedekind domains and to describe the spectrum and the valuation spectrum. This contribution focuses on nonstandard {\em number theory}; subsequent work will treat the {\em geometric} and {\em arithmetic} aspects.

In Section 2 we review the relevant facts on lattices, filters and ideals. We relate ideals of rings to ideals of lattices. Then we briefly introduce nonstandard extensions and recall some definitions and principles from nonstandard analysis.  

Section 3 contains a description of ideals in internal Dedekind rings using lattice and valuation theory. These rings have a rich ideal structure. We show 
how lattices and filters can be used to to classify general ideals, prime ideals and maximal ideals.
An important property is that internal Dedekind rings are Pr\"ufer rings. Each prime ideal is contained in exactly one maximal ideal. We consider the corresponding valuation rings and their residue fields and determine the value groups. 
The chain of prime ideals in a given maximal ideal corresponds to the ordered set of subgroups of the value group. We describe the spectrum and the maximum spectrum using lattices, ultrafilters and value groups.

In Section 4 the Riemann-Zariski space and the valuation spectrum of internal Dedekind rings and their quotient fields is investigated.  
It is also shown that Archimedean absolute values can be transformed into true valuations.

\section{Lattice Theory and Nonstandard Extensions}

\subsection{Lattices, Filters and Ideals}
\label{lattices}
We need the notion of {\em lattices} and {\em filters} and recall some definitions and facts (see for example \cite{birkhoff1948lattice}, \cite{gratzer2011lattice}). 
A {\em lattice} is a  partially ordered set $\mathcal{L}$ such any two elements  have a supremum and an infimum. If one defines for $a,b\in\mathcal{L}$
$$ a\lor b :=\sup\{a,b\} \text{ and } a\land b:=\inf\{a,b\} , $$ 
we get an algebraic structure which is associative, commutative and satisfies the absorption laws
$$\forall a,b\in\mathcal{L}: a\lor (a\land b)= a \text{ and } a\land (a\lor b)=a$$ 
On the other hand, each set with two operations $(\mathcal{L},\lor,\land)$ with associativity, commutativity and absorptions laws defines a lattice by 
$$ a\leq b :\Leftrightarrow a\land b = a \ , $$
or equivalently
$$ a\leq b :\Leftrightarrow a\lor b = b \ .$$

If we reverse the order of a lattice $\mathcal{L}$, respectively if we interchange the two  operations $\lor$ and $\land$, we get again a lattice $\mathcal{L}^{\lor}$ which is called the {\em dual lattice}. 

A subset $\mathcal{F}\subset \mathcal{L}$ is called {\em filter} if the following holds:
\begin{enumerate}
\item $\forall a,b \in \mathcal{F} : a\land b \in\mathcal{F}$, and
\item $\forall a\in \mathcal{F} \;\forall b\in\mathcal{L}: a\lor b \in\mathcal{F}$.
\end{enumerate}

In terms of the order structure this can be formulated as follows:

\begin{enumerate}
\item $\forall a,b \in \mathcal{F} : \inf\{a,b\}\in\mathcal{F}$, and
\item $\forall a\in \mathcal{F}\; \forall b\in\mathcal{L}: a \leq b \Rightarrow b\in\mathcal{F}$. 
\end{enumerate}

A filter is called {\em principal} if it has a least element. Otherwise, the filter is called {\em free}.

The dual notion of a filter is that of an {\em ideal}. A subset $\mathpzc{A}\subset \mathcal{L}$ is called {\em ideal} if the following holds:
\begin{enumerate}
\item $\forall a,b\in\mathpzc{A}: a\lor b \in\mathpzc{A}$, and 
\item $\forall a\in\mathpzc{A}\;\forall b\in\mathcal{L}: a\land b\in\mathpzc{A}$.
\end{enumerate}

We see that $\mathpzc{A}\subset \mathcal{L}$ is an ideal if and only if $\mathpzc{A}\subset \mathcal{L}^{\lor}$ is a filter.

The sets of filters and ideals of a lattice is naturally ordered by inclusion and gives again a lattice which we denote by $\mathpzc{Filt}(\mathcal{L})$ and $\mathpzc{Id}(\mathcal{L})$.

 A filter $\mathcal{F}\subset \mathcal{L}$ is called {\em proper} if $\mathcal{F}\neq\mathcal{L}$. A proper filter $\mathcal{F}$ is called {\em prime filter} if the following
holds:
$$\forall a,b\in\mathcal{F}: a\lor b \in\mathcal{F} \Rightarrow a\in\mathcal{F} \text{ or } b\in\mathcal{F}$$

  A proper filter is called {\em maximal filter} or {\em ultrafilter} if it is maximal among the proper filters. The dual notions are {\em proper ideal}, {\em prime ideal} and {\em maximal ideal}.

A lattice is called {\em distributive} if
$$\forall a,b,c\in \mathcal{L}: a\land (b\lor c)=(a\land b)\lor(a\land c)$$
or equivalently
$$\forall a,b,c\in \mathcal{L}: a\lor (b\land c)=(a\lor b)\land(a\lor c).$$

A lattice is called {\em bounded} if there exists a minimal element $0\in\mathcal{L}$ and a maximal element $1\in\mathcal{L}$. Furthermore, a bounded lattice is called {\em complemented} if
$$\forall a\in\mathcal{L}\;\exists b\in\mathcal{L}: a\land b = 0 \text{ and }  a \lor b = 1.$$
A complemented distributive lattice is a {\em Boolean Algebra}.

A lattice is called {\em relatively complemented} if for all $a,b\in\mathcal{L}$ the interval $[a,b]:=\{c\in\mathcal{L}\ \vert\  a\leq c\leq b\}$ is complemented. 
If a lattice is distributive, then the complement of an element is unique. 
A distributive lattice is relatively complemented if and only if every prime ideal is maximal \cite{gratzer1958ideal}.

For any set $M$, the power set $\mathcal{P}(M)$ and the set of finite subsets $\mathcal{P}^{fin}(M)$ ordered by inclusion are examples of distributive 
and relatively complemented lattices with a minimal element.  $\mathcal{P}(M)$ is a Boolean algebra.

Let $\mathpzc{Id}(R)$ be the set of ideals of a ring $R$. With the natural inclusion as ordering $\mathpzc{Id}(R)$ becomes a lattice, where $\lor$ is the sum and $\land$ is the intersection of ideals. By $\mathpzc{Id}_{fg}(R)$ we denote the set of finitely generated ideals. Suppose that $R$ is a coherent ring. Then the intersection of two finitely generated ideals is again finitely generated (see \cite{wisbauer} 26.6) and $\mathpzc{Id}_{fg}(R)\subset \mathpzc{Id}(R)$ is a sublattice. Quite useful in what follows is the map
$$\mathpzc{Id}(R)\rightarrow \mathpzc{Id}(\mathpzc{Id}_{fg}(R))$$
which sends an ideal $\mathpzc{a}\subset R$ to the ideal $\{\mathfrak{b}\subset\mathfrak{a}\ \vert\ \mathfrak{b}\in\mathpzc{Id}_{fg}(R)\}$ in the 
lattice $\mathpzc{Id}_{fg}(R)$.
This map is an {\em isomorphism of lattices} and the inverse map sends an ideal $\mathpzc{A}\subset\mathpzc{Id}_{fg}(R)$ to the ideal $\mathfrak{a}=\cup_{\mathfrak{b}\in\mathpzc{A}} \mathfrak{b} \subset R$. 
So an ideal in $R$ corresponds to an ideal in the lattice of finitely generated ideals. 
This fact is particularly useful in our context because finitely generated ideals of an internal ring are automatically {\em internal} (see Subsection \ref{sec:nse} below). 

There is a one-to-one correspondence between maximal ideals in $R$ 
and maximal ideals in the lattice $\mathpzc{Id}_{fg}(R)$. Furthermore, a prime ideal $ \mathfrak{p} \subset R$ yields a prime ideal in $\mathpzc{Id}_{fg}(R)$, since 
$\mathfrak{a}_1 \cap \mathfrak{a_2} \subset \mathfrak{p}$ gives $\mathfrak{a}_1 \cdot \mathfrak{a_2} \subset \mathfrak{p}$
and therefore $\mathfrak{a}_1 \subset \mathfrak{p}$ or $\mathfrak{a_2} \subset \mathfrak{p}$. On the other hand, there exist lattice prime ideals
with the property that the associated ideal in $R$ is not prime. For example, let $R=\zz$ and $p$ a prime number. Then the {\em lattice ideal} of all ideals
contained in $(p^n)$ is prime for any $n \in \nn$. This is due to the fact that lattice primality is based on the intersection of ideals and not on the product. 

Let $\mathpzc{Id}^{\times}(R)$ be the set of nonzero ideals and $\mathpzc{Id}_{fg}^{\times}(R)$ the set of nonzero finitely generated ideals.
If $R$ is a domain then these sets are sublattices and the above statements also hold for these lattices.

In general, the lattices $\mathpzc{Id}(R)$ and $\mathpzc{Id}_{fg}(R)$ are not distributive, but they are if and only if  $R$ is a Pr\"ufer ring \cite{jensen}.

\subsection{Nonstandard Extensions}
\label{sec:nse}

We briefly recall some notions and prerequisites from nonstandard analysis (see for instance \cite{albeverio1986}, \cite{lr}, \cite{goldblatt}, \cite{vaeth} for general expositions of the topic). 

The original definition of nonstandard extensions by A. Robinson is based on model theory \cite{robinson1996}. An explicit construction using {\em ultrapowers} was given by W.A.J. Luxemburg \cite{luxemburg}. Classically, sequences of real numbers modulo a free (non-principal) ultrafilter define the {\em hyperreal numbers} $\nsr$, which extend $\rr$ and contain additional infinitesimal and infinite numbers. It can be shown that the ultrapower construction can be applied to almost any mathematical object $X$ that is contained in a {\em superstructure} $\widehat{S}$ above some base set $S$. There exists an {\em embedding} map $*:\ \widehat{S} \rightarrow {\widehat{^*S}}$ from the superstructure over $S$ to the superstructure over $^*S$, which behaves in a {\em functorial} way. For details we refer to \cite{serpe}.

For our purposes, we fix a base set $S$ and a nonstandard embedding $*:\ \widehat{S} \rightarrow {\widehat{^*S}}$, which is at least countable saturated. We do not require a particular construction of $*$, but remark that such an embedding can be constructed by choosing a free
ultrafilter on $\nn$ 
and mapping $X \in {\widehat{S}}$ to the ultrapower ${^*X}:=X^{\nn} / \thicksim$ . The ultrapower is defined as the cartesian product $X^{\nn}$ modulo the ultrafilter. Two sequences are equivalent iff they are identical on an index set which is contained in the ultrafilter, and we say they agree {\em almost always}. 
If $X \subset S$ then the ultraproduct $^*X$ obviously defines a subset of $^*S$. For a general object $X \in {\widehat{S}}$,  one uses the Mostowski collapsing function (see \cite{loeb} 2.5) to show that the ultraproduct ${^*X}$ can be viewed as an element of the superstructure ${\widehat{^*S}}$ over ${^*S}$.
Elements $Y \in {\widehat{^*S}}$ are called {\em standard} if $Y={^*X}$ for $X \in {\widehat{S}}$. We call ${^*X}$ the {\em nonstandard extension}  of $X$. Note that ${^*X}$ is standard in ${\widehat{^*S}}$ but nonstandard relative to $\widehat{S}$. Furthermore, a set $Y$ is called {\em internal} if 
$Y \in {^*X}$ for $X \in {\widehat{S}}$. Internal sets are given by ultraproducts $\prod_{i\in\nn}X_i /\sim\,$, where $X_i \in X$ for each $i \in \nn$.  It is not difficult to see that there exist elements of ${\widehat{^*S}}$, which are not internal and therefore called {\em external}.  
Elements in $\,^*\mathcal{P}^{fin}(M)$ for sets $M \in {\widehat{S}}$ are called {\em hyperfinite} or $*$-finite. Hyperfinite sets are internal and can be represented by ultraproducts of finite sets. If $Y$ is an internal set, then the set of internal subsets $\,^*\mathcal{P}(Y)$ and the set of $^*$-finite subsets $\,^*\mathcal{P}^{fin}(Y)$ are distributive, relatively complemented lattices with a minimal element.

This contribution uses nonstandard extensions  of sets, groups, rings and fields. 
The nonstandard embedding $*$ satisfies a number of important principles for which we refer to the literature:  the {\em Transfer Principle}, {\em Countable Saturation}, {\em Countable Comprehension} and the
{\em Permanence Principle}. 

\section{Ideals in Nonstandard Dedekind Rings}
\label{fields}

\subsection{Internal Dedekind Rings}
Our notation is as follows: $A$ is a Dedekind domain with quotient field $K$, for example a number field $K$ with ring of integers $A=\oo_K$. 

$|\ \ |_v  : K \rightarrow \mathbb{R}_{\geq 0}$ 
denotes an absolute value of $K$. An absolute value is called {\em non-Archimedean}, if it satisfies the strong triangle inequality. Otherwise, the 
absolute value is called {\em Archimedean}  and one can show that all Archimedean absolute values arise by embedding $K$ into $\rr$ or
$\cc$ and taking the standard norm (Big Ostrowski Theorem).

\begin{definition} 
Let $B$ be an internal domain with quotient field $L$. 
\begin{enumerate}
\item The $*$-integral closure of $B$ in $L$ is the ring of all $x \in L$ which are zeros of
 a polynomial
$f(x)=\displaystyle\sum_{i=0}^N a_i x^i$ of hyperfinite degree $N \in \nsn$. 

\item $B$ is 
$*$-integrally closed if $B$ is the $*$-integral closure of $B$ in $L$.  

\item $B$ is called internal Dedekind ring or nonstandard Dedekind ring if $B$ is $*$-integrally closed and every internal prime ideal is maximal.\hfill$\lozenge$
\end{enumerate}
\end{definition}

The definitions of $*$-integral closure, $*$-integrally closed and internal Dedekind ring are hence obtained by transferring the corresponding standard definitions. For a standard Dedekind domain $A$, the nonstandard extension ${^*A}$ is an internal Dedekind ring. We will denote internal Dedekind rings by $B$ and their quotient field by $L$, for example $B={^*A}$ and $L={^*K}$. 

\begin{example} Let $p$ be a standard prime, $N \in \nsn$ be an infinite number and $\zeta_{p^N}$ a primitive $p^N$-th root of unity. 
Then $B=\nsz[\zeta_{p^N}]$ is an internal Dedekind ring which contains all standard Dedekind rings $\zz[\zeta_{p^n}]$ for $n \in \nn$.\hfill$\lozenge$
\end{example}

Similar as above, a nonstandard absolute value $|\ \ |_v : L \rightarrow \nsr_{\geq 0}$ is called {\em non-Archimedean} if the strong triangle inequality holds, and otherwise  {\em Archimedean}. 

Valuations $v: R \rightarrow G \cup \{\infty\}$ where $R$ is a ring (or field) and $G$ an ordered group
are written additively in this paper (exponential or Krull valuation). Note that the multiplicative notation is also used in the recent literature.
We do not restrict the value group $G$ in this paper and do {\em not require that $v$ is standard or internal}. We will see below that interesting 
constructions arise from external valuations.

\begin{remark} 
We give an explicit construction of internal fields and domains. Let $I$ be an index set, e.g.\ $I=\nn$, and $(K_i)_{i\in I}$ a family of fields, e.g.\ $K_i=K$ for a number field $K$. Set $R = \prod_{i \in I} K_i$. For $x=(x_i) \in R$, let $Z(x)=\{ i \in I \ |\ x_i=0\}$. An ideal $\aaa \subset R$ yields a filter
$\mathcal{F}=\{ Z(x)\ |\ x \in \aaa \}$ on $I$. Obviously, $\mathcal{F}$ is upward closed. We claim  $Z(x) \cap Z(y) \in \mathcal{F}$ for $x,y \in \aaa$. One can assume that $x$ and $y$ is a sequence of zeros and ones. Set $w=x+(1-x)y \in \aaa$. Then $Z(w)=Z(x) \cap Z(y) \in \aaa$.
Conversely, every filter $\mathcal{F}$ on $I$ gives an ideal $\aaa=\{ x \in R \ |\ Z(x) \in \mathcal{F} \} \subset R$. Indeed, $rx \in \aaa$ if $r \in R$ and $x \in \aaa$. Furthermore, $x+y \in \aaa$ for $x,y \in \aaa$ since $Z(x) \cap Z(y) \subset Z(x+y)$.
This defines a bijection between ideals of $R$ and filters on $I$. Furthermore, it gives a bijection between maximal ideals and maximal filters. If $\aaa$ is maximal then the quotient $L=R/\aaa$ is an {\em ultraproduct} of the fields $K_i$. The ultraproduct is also denoted by $\left(\prod_i K_i \right)/ \mathcal{F}$.

Suppose that our nonstandard embedding $*$ is defined by taking ultraproducts using a fixed free ultrafilter $\mathcal{F}$ on $I$, for example on $I=\nn$. 
Then all internal objects arise by taking ultraproducts with respect to $\mathcal{F}$.
Let $(A_i)_{i \in I}$ be a family of domains with quotient fields $K_i$. Then the ultraproduct $B=\left(\prod_{i \in I} A_i \right)/ \mathcal{F}$ is an internal domain.
As above, two elements of $\prod_{i \in I} A_i $ are identified in $B$ if they coincide on an index set which is contained in $\mathcal{F}$.\hfill$\lozenge$
\end{remark}

The nonstandard extension of a number field and its ring of integers is particularly interesting.  The following Proposition shows that the different integral closures coincide.

\begin{prop} Let $K$ be a number field and $\oo_K$ the ring of integers. Let $\oo_{^*K}$ 
be the standard integral closure  and ${^*\oo}_{^*K}$ the $*$-integral closure of $\nsz$ in $^*K$. 
Then
$$^*(\oo_K)=({^*\oo})_{^*K}=\oo_{^*K} \ . $$ 
\end{prop}
\begin{proof} A general statement on the $*$-value of sets defined by formulas (see \cite{lr} 7.5) implies 
\begin{align*}
{^* (\oo_K)} & = {^*\{ x \in K \ |\ \exists n \in \nn \ \exists a_0,\dots,a_{n-1} \in \zz \ :  \ x^n + a_{n-1}x^{n-1} + \dots + a_0 = 0 \} } \\
 & = \{ x \in {^*K} \ |\ \exists n \in \nsn \ \exists a_0,\dots,a_{n-1} \in \nsz \ : \ x^n + a_{n-1}x^{n-1} + \dots + a_0 = 0 \} \\
 & = ({^*\oo})_{^*K}
\end{align*} 
Since $[K:\qq]=[{^*K}:{^* \qq}]$ is finite, the minimal polynomial of an element in ${^*K}$ that is integral over $\nsz$ must have finite degree $n \in \nn$. This shows $({^*\oo})_{^*K}=\oo_{^*K}$.
\end{proof}

We return to the general case of an internal Dedekind ring $B$.  Let $\mathcal{M}$ be the set of internal maximal ideals of $B$. We denote by ${^*\bigoplus}_{\mathcal{M}} {^*\zz}$ the group of internal Weil divisors of $B$. This is the set of internal functions from $\mathcal{M}$ to $\nsz$ with hyperfinite support, i.e.\ internal functions which are zero outside a hyperfinite subset of $\mathcal{M}$.

The next Proposition follows from properties of standard Dedekind rings and application of the Transfer Principle.

\begin{prop} Let $B$ be an internal Dedekind ring and $L$ its quotient field. 
\begin{enumerate}
\item An ideal of ${B}$ is internal if and only if it is finitely generated (by at most two elements). 
\item All internal ideals $\neq (0)$ are invertible.
\item All internal prime ideals $\neq (0)$ are maximal.
\item Every internal maximal ideal on $B$ induces a valuation of ${L}$ with value group $\nsz$ and there is a one-to-one correspondence between the set 
of internal valuations of $L$ and the set of internal maximal ideals of $B$. 
\item
The nonzero internal fractional ideals of $L$  factor into a hyperfinite product of hyperfinite powers of internal maximal ideals. There is a natural isomorphism between the nonzero internal fractional ideals and the group of internal Weil divisors ${^*\bigoplus}_{\mathcal{M}} {^*\zz}$.
\item The nonzero internal ideals of $B$ form a lattice that is isomorphic to $({^*\bigoplus}_{\mathcal{M}} {^*\mathbb{N}})^{op}$. The intersection of ideals corresponds to the maximum value and the sum to the minimum value at each internal maximal ideal.
\end{enumerate}
\label{ideals}
\end{prop}

The internal valuations correspond to finitely generated maximal ideals, but we will see below that there are many additional {\em external} valuations of $L$. The following fact is important for the classification of arbitrary prime ideals and permits us to leverage valuation theory.

\begin{cor} Let $B$ be an internal Dedekind ring. Then $B$ is a Pr\"ufer ring, i.e.\ finitely generated ideals are invertible. For every prime ideal $\pp$ the localisation $B_{\pp}$ is a valuation ring.
\label{pruefer}
\end{cor}
\begin{proof} This follows from  Proposition \ref{ideals} and \cite{jensen}. 
\end{proof}
Note that $B_{\pp}$ is a valuation ring even if $\pp$ is external, i.e.\ if $\pp$ is not finitely generated.

\subsection{Ideals of Internal Dedekind Rings and Filters on Lattices}
In the following, we want to investigate the ideal structure of internal Dedekind rings $B$, for example $B={^*A}$ for a standard Dedekind ring $A$.  There are three types of ideals of ${B}$:
\begin{enumerate}
\item {\em Standard} ideals ${^*\aaa} \subset {B}$ are extensions of ideals $\aaa \subset B$. They are finite products of standard maximal ideals.
\item {\em Internal} ideals $\bbb \subset {B}$ are hyperfinite products of internal maximal ideals. 
 There exist internal ideals which are not standard. 
\item {\em External} ideals $\ccc \subset {B}$. We will see below that there exist many non-internal ideals and external ideals possess a rich structure.
\end{enumerate}

Our aim is to give a classification of ideals in $B$ and to extend the results in
\cite{cherlindedekind}, \cite{cherlin} and \cite{klingen1975}. The
description uses lattices, ideals in lattices and filters.

For the remainder of this section we fix the following notation: {\em $B$ denotes an internal Dedekind ring and $\mathcal{M}$ is the set of internal maximal ideals of $B$. }

Since Pr\"ufer rings are coherent and finitely generated ideals coincide with internal ideals, we have an
isomorphism of lattices as explained in Section \ref{lattices}:

$$\mathpzc{Id}(B)\stackrel{\sim}{\longrightarrow}\mathpzc{Id}(\mathpzc{Id}_{fg}(B))=\mathpzc{Id}(^*\mathpzc{Id}(B))$$

One can also remove the zero ideal:

$$\mathpzc{Id}^{\times}(B)\stackrel{\sim}{\longrightarrow}\mathpzc{Id}(\mathpzc{Id}_{fg}^{\times}(B))=\mathpzc{Id}(^*\mathpzc{Id}^{\times}(B))$$

 \begin{definition} Let $R$ be an internal ring and $\aaa \subset R$ an ideal, which is not necessarily internal. 
 Then the $*$-radical of $\aaa$, denoted by $\srad(\aaa)$, is defined as the set of all $x \in R$ such that $x^n \in \aaa$ for some $n \in \nsn$.
 $\aaa$ is called {\em $*$-radical} if  $\srad(\aaa)=\aaa$.
 \end{definition}
 
 \begin{remark} Obviously, $\rad(\aaa) \subset \srad(\aaa)$. Maximal ideals are $*$-radical, but there exist radical ideals and even prime ideals which are not $*$-radical.   
 \end{remark}

Because the definition of the $*$-radical is internal we see

\begin{lemma}
The $*$-radical of an internal ideal is internal.
\end{lemma}

The Lemma shows the following:
\begin{prop}\label{radid}
The $*$-radical ideals of $B$ correspond under the
isomorphism
$\mathpzc{Id}(B)\xrightarrow{\sim}\mathpzc{Id}(^*\mathpzc{Id}(B))$ to the
ideals which are generated by internal $*$-radical ideals.
\end{prop}

\begin{prop}
There is a natural isomorphism of lattices
$$\mathpzc{Id}^{\times}(B) \stackrel{\sim}{\longrightarrow}\mathpzc{Filt}\left({^*\bigoplus}_{\mathcal{M}}
{^*\mathbb{N}}\right).$$
\end{prop}
\begin{proof} $B$ is an internal Dedekind ring. Then Proposition \ref{ideals}(f) gives an
isomorphism of lattices:
$$\mathpzc{Id}_{fg}^{\times} (B) \stackrel{\sim}{\longrightarrow} \left({^*\bigoplus}_{\mathcal{M}}
{^*\mathbb{N}}\right)^{op}$$
General ideals correspond to ideals of finitely generated ideals (see Section \ref{lattices}). 
This implies the following isomorphisms: 

$$\mathpzc{Id}^{\times}(B) \stackrel{\sim}{\longrightarrow} 
\mathpzc{Id}(\mathpzc{Id}_{fg}^{\times}(B))  \stackrel{\sim}{\longrightarrow} \mathpzc{Id}\left(\left({^*\bigoplus}_{\mathcal{M}} {^*\mathbb{N}}\right)^{op}\right) 
\stackrel{\sim}{\longrightarrow} \mathpzc{Filt}\left({^*\bigoplus}_{\mathcal{M}}
{^*\mathbb{N}}\right)$$
\end{proof}

Under the above isomorphism, we call the image of 
a $*$-radical ideal a {\em $*$-radical filter}. Such a filter is generated by  elements in ${^*\bigoplus}_{\mathcal{M}} \{0,1\}$. 
Furthermore, every filter  $\mathcal{F} \subset {^*\bigoplus}_{\mathcal{M}}
{^*\mathbb{N}}$ has an associated $*$-radical filter $\srad(\mathcal{F})$ which can be constructed using the maps $Supp$ and $Ch$ (see Theorem \ref{prop-ideals} and the Definitions before that Theorem).

For the classification of $*$-radical ideals it is useful to consider the lattice
$$\mathcal{L}:= {^*\mathcal{P}}_{fin}(\mathcal{M})$$ 
of hyperfinite subsets of $\mathcal{M}$.

Instead of working with {\em hyperfinite} subsets one could also use the lattice
${^*\mathcal{P}}(\mathcal{M})$ of {\em internal} subsets. 
${^*\mathcal{P}}(\mathcal{M})$ and $\mathcal{L}$ are both distributive, relatively complemented lattices.
In the literature, both lattices are used to classify ideals (see \cite{klingen1975}, \cite{cherlin}, \cite{olberdingshapiro}). The following Proposition shows that these approaches are equivalent.

\begin{prop} A filter $\mathcal{F} \subset {^*\mathcal{P}}(\mathcal{M})$ is called {bounded} if it contains a hyperfinite set. There is a bijection between bounded proper filters on ${^*\mathcal{P}}(\mathcal{M})$ and proper filters on $\mathcal{L}$. The bounded ultrafilters on ${^*\mathcal{P}}(\mathcal{M})$ correspond to ultrafilters on $\mathcal{L}$.
\label{filters}
\end{prop}
\begin{proof} Let $\mathcal{F} \subset {^*\mathcal{P}}(\mathcal{M})$ be  a bounded proper filter. Then $\mathcal{F}_{\mathcal{L}}=\mathcal{F} \cap \mathcal{L}$ is non-empty and obviously a filter on $\mathcal{L}$.  $\mathcal{F}_{\mathcal{L}}$ is proper since every internal set contains a hyperfinite set. Conversely, a proper filter $\mathcal{F}_{\mathcal{L}}$ on $\mathcal{L}$ yields a proper bounded filter on  ${^*\mathcal{P}}(\mathcal{M})$  by adding all internal supersets of sets in $\mathcal{F}_{\mathcal{L}}$. We show that the maps are inverse to each other. Let $\mathcal{F}$ be a bounded proper filter on ${^*\mathcal{P}}(\mathcal{M})$ and $F \in \mathcal{F}$. Then $H \in \mathcal{F}$ for some hyperfinite set $H$. Thus $F \cap H \in \mathcal{F}$ and $F \cap H \in \mathcal{F}_{\mathcal{L}}$. But this implies that $F$ is contained in the filter on ${^*\mathcal{P}}(\mathcal{M})$ that is associated to $\mathcal{F}_{\mathcal{L}}$. The converse concatenation obviously gives the identity map. The bijection preserves inclusions so that maximal filters on ${^*\mathcal{P}}(\mathcal{M})$ correspond to maximal filters on $\mathcal{L}$.
\end{proof}

\begin{example}
\begin{enumerate}
\item Let $H \subset \mathcal{M}$ be any hyperfinite non-empty set. The the set of internal (respectively hyperfinite) supersets of $H$ defines a principal filter 
$\mathcal{F}$ on ${^*\mathcal{P}}(\mathcal{M})$
(respectively on $\mathcal{L}$). $\mathcal{F}$ is an ultrafilter iff $H$ contains exactly one maximal ideal.
\item Let $H \subset \mathcal{M}$ be an infinite hyperfinite set. Then define the following filter $\mathcal{F}$ on $\mathcal{L}$: $F \in \mathcal{F}$ if there exists a finite set $H_0 \subset \mathcal{M}$ such that $H \setminus H_0 \subset F$. $\mathcal{F}$ is a free filter since it does not have a least element. Furthermore, $\mathcal{F}$ can be extended to a free ultrafilter on $\mathcal{L}$ and on ${^*\mathcal{P}}(\mathcal{M})$.
\end{enumerate}
\label{non-internal}
\end{example}

We study the relation between the lattices . There is 
are natural homomorphisms of lattices ${^*\bigoplus}_{\mathcal{M}} {^*\mathbb{N}}$ and $\mathcal{L}$.
$$supp:\ {^*\bigoplus}_{\mathcal{M}} {^*\mathbb{N}}\longrightarrow
\mathcal{L}$$
defined by $supp(f)=\{m\in\mathcal{M}\vert\ f(m)\neq 0\}$, 
and 
$$ch:\ \mathcal{L} \longrightarrow
{^*\bigoplus}_{\mathcal{M}} {^*\mathbb{N}} $$ 
where $ch(S)(m)= 1 $ if $m \in S$ and $0$ otherwise.

For all $S\in \mathcal{L}$ we have $supp(ch(S))=S$.
These maps induce lattice homomorphisms
$$Supp:\mathpzc{Filt}\left( {^*\bigoplus}_{\mathcal{M}} {^*\mathbb{N}} \right)\longrightarrow
\mathpzc{Filt}(\mathcal{L}) $$
defined by $Supp(\mathcal{F})=\{supp(f) \vert\ f \in \mathcal{F}\}$
and 
$$Ch:\mathpzc{Filt}(\mathcal{L}) \longrightarrow
 \mathpzc{Filt}\left( {^*\bigoplus}_{\mathcal{M}} {^*\mathbb{N}} \right)$$
defined by $Ch(\mathcal{S})(m):=\{f\in{^*\bigoplus}_{\mathcal{M}} {^*\mathbb{N}} \
\vert\  \exists S\in \mathcal{S}: ch(S)\leq f\}$
which maps proper filters to proper filters.
Note that for $\mathcal{S}\in
\mathpzc{Filt}(\mathcal{L})$  
the filter in $\mathpzc{Filt}( {^*\bigoplus}_{\mathcal{M}}
{^*\mathbb{N}})$ generated by all $ch(S)$ for $S\in\mathcal{S}$ is
  precisely $Ch(\mathcal{S})$.

The following Theorem describes the properties of the lattice homomorphisms $Supp$ and $Ch$ and gives the relation between prime filters and ultrafilters on the lattices ${^*\bigoplus}_{\mathcal{M}} {^*\mathbb{N}}$ and $\mathcal{L}$.

\begin{theorem}\label{prop-ideals}
Let $\mathcal{S}\in \mathpzc{Filt}(\mathcal{L})$ and $\mathcal{F}\in\mathpzc{Filt}({^*\bigoplus}_{\mathcal{M}} {^*\mathbb{N}}) $. Then
\begin{enumerate}
\item  $(Supp \circ Ch)(\mathcal{S})=\mathcal{S}$.
\item  $(Ch \circ
  Supp)(\mathcal{F}) =\srad(\mathcal{F} )$.
\item $Supp(\mathcal{F})$ is a prime filter if and only if $\mathcal{F}$ is a prime filter.
\item $Ch(\mathcal{S})$ is a prime filter if and only if $\mathcal{S}$ is a prime filter.
In this case, $\mathcal{S}$ and $Ch(\mathcal{S})$ are maximal filters.
\item $\mathcal{F}$ is a prime filter if and only if $\mathcal{F}$ is contained in exactly one maximal filter.
In this case, the associated ultrafilter is $\srad(\mathcal{F})$.
\item $Supp$ and $Ch$ provide a one-to-one correspondence between the maximal filters on $\mathcal{L}$ and the maximal filters on 
${^*\bigoplus}_{\mathcal{M}} {^*\mathbb{N}}$.
\end{enumerate}
\end{theorem}
\begin{proof}
(a) follows from the definition.

We have $f \in {^*Rad} (\mathcal{F})$ iff there exists $n \in \nsn$ such that $n \cdot f \in \mathcal{F}$. For any $f \in {^*Rad} (\mathcal{F})$, 
 $ch(supp(n \cdot f))=ch(supp(f)) \leq f$ holds for all $n \in \nsn$. Furthermore, for any $f \in \mathcal{F}$ there exists $n \in \nsn$ such that $f \leq n \cdot ch(supp(f))$. This yields (b).
 
 (c) Suppose $\mathcal{F}$ is a prime filter and $S_1 \cup S_2 =  supp(f)$ with $S_1, S_2 \subset \mathcal{L}$ and $f \in \mathcal{F}$. Define $f_1, f_2 \in {^*\bigoplus}_{\mathcal{M}} {^*\mathbb{N}}$ by $f_i(m)=f(m)$ if $m \in S_i$ and $f_i(m)=0$ otherwise, where $i=1,2$. Then $supp(f_i)=S_i$ and $f_1 \vee f_2 = f$. Since $\mathcal{F}$ is prime, we have $f_1 \in \mathcal{F}$ or $f_2 \in \mathcal{F}$ and hence $S_1 \in Supp(\mathcal{F})$ or $S_2 \in Supp(\mathcal{F})$. 

Conversely, assume that $Supp(\mathcal{F})$ is a prime filter and $f_1 \vee f_2 = f \in \mathcal{F}$. Note that $f_1 \vee f_2$ is defined by taking the maximum value at each $m \in \mathcal{M}$. Define $f'_1$ and $f'_2$ by $f'_1(m)=f_1(m)$ if $f_1(m)>f_2(m)$ and $f'_1(m)=0$ otherwise and set 
$f'_2(m)=f_2(m)$ if $f_1(m) \leq f_2(m)$ and $f'_2(m)=0$ otherwise. Note that $f'_1 \leq f_1$ and $f'_2 \leq f_2$.
We have $f_1 \vee f_2 = f'_1 \vee f'_2 = f$ and hence $supp(f'_1) \cup supp(f'_2) = supp(f)$.
This implies $supp(f'_i) \in Supp(\mathcal{F})$ for $i=1$ or $i=2$. Assume that $g \in \mathcal{F}$ with $supp(g)=supp(f'_1)$. We want to show
that $f'_1 \in \mathcal{F}$ and therefore $f_1 \in \mathcal{F}$. To this end we use the fact that 
$$g \wedge (f'_1 \vee f'_2) = (g \wedge f'_1) \vee (g \wedge f'_2) \in \mathcal{F}$$
which follows from our assumption and the distributivity of the lattice ${^*\bigoplus}_{\mathcal{M}} {^*\mathbb{N}}$.
By construction, $supp(g) \cap supp(f'_2) = supp(f'_1) \cap supp(f'_2) = \varnothing$. Hence $g \wedge f'_2$ is the zero function
which implies $g \wedge f'_1 \in \mathcal{F}$. Since $g \wedge f'_1 \leq f'_1 \leq f_1$ we obtain $f_1 \in \mathcal{F}$ which concludes the proof of (c).

The first part of (d) follows from (a) and (c). If $\mathcal{S}$ is prime then it is maximal since $\mathcal{L}$ is a relatively complemented lattice.
By construction, $Ch(\mathcal{S})$ is the largest filter in ${^*\bigoplus}_{\mathcal{M}} {^*\mathbb{N}}$ with support equal to $\mathcal{S}$. Thus  $Ch(\mathcal{S})$ must be maximal and the second part of (d) follows.

(e) If $\mathcal{F}$ is prime then $\mathcal{F}$ is contained in the maximal filter $Ch(Supp(\mathcal{F}))=\srad(\mathcal{F})$. Any other maximal filter $\mathcal{G}$ which contains $\mathcal{F}$ must be $*$-radical and satisfy $Supp(\mathcal{F})=Supp(\mathcal{G})$ since $Supp(\mathcal{F})$ is maximal.  From (b) we obtain $\mathcal{F}=\mathcal{G}$. \\
Conversely, if $\mathcal{F}$ is not prime then by (c) $\mathcal{S}=Supp(\mathcal{F})$ is not prime either. So there
exist $S_1, S_2 \in \mathcal{S}$ such that $S_1 \cup S_2 \in \mathcal{S}$, but $S_1, S_2 \notin \mathcal{S}$. Then there exist prime (i.e. maximal) filters $\mathcal{S}_1 \supset \mathcal{S}$ and $\mathcal{S}_2 \supset \mathcal{S}$ such that $S_2 \notin \mathcal{S}_1$ and $S_1 \notin \mathcal{S}_2$ (see \cite{gratzer2011lattice} Corollary 116). Since $\mathcal{S}_1$ and $\mathcal{S}_2$ are prime filters that contain $S_1 \cup S_2$, one has $S_1 \in \mathcal{S}_1$ and $S_2 \in \mathcal{S}_2$. Hence $\mathcal{S}$ is contained in two different 
maximal filters. We conclude that $\mathcal{F}$ is contained in different maximal filters $Ch(\mathcal{S}_1)$ and $Ch(\mathcal{S}_2)$. This shows (e).

(f) Let $\mathcal{F}$ be a maximal filter on ${^*\bigoplus}_{\mathcal{M}} {^*\mathbb{N}}$. Then $Supp(\mathcal{F})$ is maximal by (c) and
$Ch(Supp(\mathcal{F}))=\mathcal{F}$ by (b). Conversely, let $\mathcal{S}$ be maximal filter on $\mathcal{L}$. It follows from (c) that
$Ch(\mathcal{S})$ is maximal and $Supp(Ch(\mathcal{S}))=\mathcal{S}$ by (a). This shows that $Supp$ and $Ch$ are inverse maps on the maximal filters. 
\end{proof}

\begin{cor}
Every nonzero prime ideal of $B$ is contained in exactly one maximal ideal. 
\end{cor}

\begin{proof}
A nonzero prime ideal in $B$ yields a prime filter in ${^*\bigoplus}_{\mathcal{M}} {^*\mathbb{N}}$ which is contained in a unique maximal filter by Theorem \ref{prop-ideals}. 
\end{proof}

\begin{remark} Prime ideals in $B$ yield a prime filter in ${^*\bigoplus}_{\mathcal{M}} {^*\mathbb{N}}$, but prime filters do not necessarily correspond to a  {\em prime} ideal in $B$ (see Section \ref{lattices}).\hfill$\lozenge$
\end{remark}

The following statement describes general filters as intersections of prime filters.

\begin{prop} Let $\mathcal{F}$ be a filter on ${^*\bigoplus}_{\mathcal{M}} {^*\mathbb{N}}$. Then $\mathcal{F}$ is the intersection of the prime filters containing $\mathcal{F}$.
\end{prop}
\begin{proof} This holds true for every distributive lattice, see \cite{gratzer1958ideal}.
\end{proof}

\begin{remark}
An intersection over a subset of prime filters containing $\mathcal{F}$ can give the same filter $\mathcal{F}$. Thus the above representation as an intersection of prime filters is not unique.\hfill$\lozenge$
\end{remark}

The above Proposition  reduces the classification of ideals in $B$ to prime filters on ${^*\bigoplus}_{\mathcal{M}} {^*\mathbb{N}}$. The corresponding ideals are those which are contained in only one maximal ideal.

\begin{cor} Every ideal in $B$ is the intersection of ideals which are contained in exactly one maximal ideal of $B$.
\end{cor}

 Theorem  \ref{prop-ideals} shows that $*$-radical and maximal ideals of $B$ can classified using the lattice $\mathcal{L}$.

\begin{cor} 
\begin{enumerate}
\item The maps $Supp$ and $Ch$ induce a one-to-one correspondence between proper nonzero $*$-radical ideals of $B$ and proper
 filters on $\mathcal{L}$. Internal ideals correspond to principal filters
 on $\mathcal{L}$.
 \item The maps $Supp$ and $Ch$ induce a one-to-one correspondence between maximal ideals of
 $B$ and ultrafilters on the lattice $\mathcal{L}$. The principal
ultrafilters correspond to internal maximal ideals, and free
ultrafilters yield external maximal ideals.

\end{enumerate}
\label{class-max-ideals}
\end{cor}

\begin{remark} We explicitly describe the bijections in Corollary \ref{class-max-ideals}: let $\aaa$ be a proper nonzero $*$-radical ideal (part a) or a maximal ideal (part b). For $a \in \aaa$, $a \neq 0$ let 
$$V(a) = \{ \pp \in \mathcal{M} \ |\ a \in \pp \}  \in \mathcal{L} . $$ 
Then $\{ V(a) \ |\ a \in \aaa,\ a \neq 0\}$ is a subbase and generates a filter $\mathcal{F}$ on $\mathcal{L}$. 

Conversely, if $\mathcal{F}$ is a given proper filter on $\mathcal{L}$ then set
$$ \aaa = \{ a \in B \ |\ V(a) \in \mathcal{F} \text{ or } a=0 \}.$$
$\aaa$ is maximal if and only if $\mathcal{F}$ is an ultrafilter.\hfill$\lozenge$
\end{remark}

We have seen above that maximal ideals in $B$ are in one-to-one correspondence to
maximal filters on $\mathcal{L}$. Internal maximal ideals correspond to principal and external ideals to free ultrafilters. The following Remark shows that the set of external ideals is quite large.

\begin{remark} Example \ref{non-internal} (b) shows that external maximal ideals exist. Let $\aaa=\prod_{\mm \in S} \mm$ where $S$ is an infinite and hyperfinite set of internal maximal ideals.
Then $\aaa$ is internal, but is contained in infinitely many external maximal ideals; in fact the set of external maximal ideals above $\aaa$ is in bijection with the set of free ultrafilters on ${^*\mathcal{P}}_{fin}(S)$. Furthermore, disjoint hyperfinite sets $S$ and $S'$  yield relatively prime ideals $\aaa$ and $\aaa'$ and the sets of maximal ideals above these ideals are disjoint.\hfill$\lozenge$
 \end{remark} 

The next Proposition describes the residue field of a maximal ideal in terms of the associated ultrafilter.

\begin{theorem} Let $\mm$ be a maximal ideal of $B$ and denote by $\kappa = B/\mm$ the residue  field. 
Let $\mathcal{F}$ be the ultrafilter on $\mathcal{L}$ which corresponds to $\mm$. 
Then $\kappa$ is isomorphic to the ultraproduct of the internal residue  fields $B/\pp$ where $\pp \in \mathcal{M}$ :
 \[ \kappa \cong \left( {^*\prod_{\pp \in \mathcal{M}}} B/\pp  \right) / \ \mathcal{F} \] 
If $\mm$ is internal, then $\mathcal{F}$ is principal and the ultraproduct collapses to the internal residue field $B/\mm$.
\label{residue}
 \end{theorem}
 
 \begin{proof} We consider the following map which is induced by diagonal embedding and projection:
 \[ \pi:\ B \longrightarrow \left( {^*\prod_{\pp \in \mathcal{M}}} B/\pp  \right) / \ \mathcal{F} \]
Two sequences are identified if they coincide on an index set $F \in \mathcal{F}$. We claim that $\ker(\pi)=\mm$. For the inclusion "$\subset$", let $a \in \ker(\pi)$ and $a \neq 0$. Then there exists a set $F \in \mathcal{F}$ such that 
 $a \in \pp$ for all $\pp \in F$. Hence $F \subset V(a)$, which implies $V(a) \in \mathcal{F}$ and therefore $a \in \mm$. For the inverse inclusion "$\supset$",
 let $a \in \mm$ and $a \neq 0$, so that $V(a) \in \mathcal{F}$. Hence $a \in \pp$ for all $\pp \in V(a)$ which yields $a \in \ker(\pi)$.\\
 The proof is completed by showing that $\pi$ is surjective. 
 Let $a \in \mm$ be any nonzero element. Then $V(a) \in \mathcal{F}$ and $\pi$ can be factorized as follows:
 \[ \pi:\ B \longrightarrow \
 {^*\prod_{\pp \in V(a)}}  B/\pp  
 \longrightarrow
\left( {^*\prod_{\pp \in \mathcal{M}}} B/\pp  \right) / \ \mathcal{F} \]
Since $V(a)$ is hyperfinite, the surjectivity  of the first map follows from the transfer of the Chinese Remainder Theorem. By definition of the reduction modulo $\mathcal{F}$, the second map must also be surjective.\end{proof}
 
\begin{remark} Consider  an internal number field $L$ and $B= \noo_L$. Then the internal residue fields of $B$ are hyperfinite and
  Theorem \ref{residue} shows that every residue field is isomorphic to an ultraproduct of finite fields for a suitable chosen ultrafilter.\hfill$\lozenge$
\end{remark}  
 
 \subsection{Ideal Class Groups and Units}
 
For an integral domain $R$ we denote by $J(R)$ the group of invertible fractional ideals, by $P(R)$ the subgroup of principal fractional ideals. The quotient  $Cl(R):=J(R)/P(R)$ is the ideal class group, i.e.\ the
group of invertible fractional ideals modulo principal ideals. 

\begin{prop} Let $L$ be an internal number field. Then
$Cl(\noo_L)$ is a hyperfinite group and $(\noo_L)^{\times}$ has a hyperfinite number of generators.
\label{hyperfiniteness}
\end{prop}
\begin{proof} This follows from the corresponding classical results (finiteness of the class group and Dirichlet's unit theorem) and application of the transfer principle.
\end{proof}

For a standard number field $L={^*K}$ we even obtain finiteness:

\begin{prop}
For a number field $K$ the canonical maps
$$Cl(\oo_K)\rightarrow Cl(\noo_K)$$
and 
$$\oo_K^{\times}\otimes_{\mathbb{Z}}\; ^*\mathbb{Z} \rightarrow
(\noo_K)^{\times} $$
are bijective.
\label{finiteness}
\end{prop}
\begin{proof} By the
finiteness of $Cl(\oo_K)$ the map 
$$Cl(\oo_K)\rightarrow {^*(Cl(\oo_K))}$$
is bijective. Now we consider the diagram
\[
\xymatrix{
0\ar[r] &  ^*P(\oo_K) \ar[r] \ar[d] & ^*J(\oo_K)  \ar[r] \ar[d]
&^*Cl(\oo_K) \ar[r] \ar[d] & 0\\
0\ar[r] & P(\noo_K) \ar[r] & J(\noo_K) \ar[r] \ar[r] & Cl(\noo_K) \ar[r]&0
}
\]
The first line is exact because $*$ is an exact functor (cp.\ \cite{serpe}). The second line is exact by definition. Further invertible fractional ideals are finitely
generated and finitely generated ideals are automatically internal. Therefore the maps $^*P(\oo_K)\rightarrow P(\noo_K)$ and $^*J(\oo_K)\rightarrow J(\noo_K)$ are bijective. It follows that  $Cl(\noo_K)={^*Cl(\noo_K)}$ is also bijective and so the first claim.

By Dirichlet's Unity theorem we know that $\oo_K^{\times}$ is a finitely
generated $\mathbb{Z}$-module. Therefore $^*(\oo_K^{\times})=(\noo_K)^{\times}$ is
a finitely generated $^*\mathbb{Z}$-module with the same
generators. This gives the second claim.\end{proof}

\subsection{Valuation Rings and Value Groups}

We study the valuation rings and value groups of internal Dedekind rings.
  Let $B$ be an internal Dedekind ring with quotient field $L$ and $\mm$ a maximal ideal of $B$. Since $B$ is Pr\"ufer, the localization $B_{\mm}$ is a valuation ring (see Corollary \ref{pruefer}). If $\mm$ is internal, then $B_{\mm}$ is a $*$-discrete (aka ultra-discrete) valuation ring with field $L$ and value group $\nsz$.
\label{loc}

We describe the {\em value group} of $B_{\mm}$ for all maximal ideals $\mm$, including the {\em external} ideals. First, we need a statement over standard Dedekind rings:

\begin{prop} Let $A$ be a Dedekind ring and $W$ a finite set of maximal ideals of $A$. Define the multiplicative set $S= \bigcap_{ \pp \in W} (A \setminus \pp)$. Then
$$ S^{-1} A = \bigcap_{\pp \in W} A_{\pp} \ . $$
\end{prop}
\begin{proof} Note that elements in $\frac{a}{b} \in S^{-1} A$  have a representation $x= \frac{a}{b}$ with $b \notin \pp$ for {\em all} $\pp \in W$. Thus the inclusion "$\subset$" is trivial.

Obviously, the nonzero ideals of $S^{-1} A$ are $\pp S^{-1} A$ where $\pp \in W$. These ideals are also maximal. For $\pp \in W$, the localization $(S^{-1} A)_{\pp}$ is isomorphic to the discrete valuation ring $A_{\pp}$. Hence $S^{-1}A$ is a Dedekind domain. Then $S^{-1} A$ is the intersection over all localizations at height 1 (i.e.\ at maximal ideals) (see \cite{matsumura} 11.5).
\end{proof}

The Transfer Principle implies the following Corollary:
\begin{cor} Let $B$ be an internal Dedekind ring with quotient field $L$ and let $F$ be a hyperfinite set of internal maximal ideals of $B$. Suppose that $x \in L$ satisfies $v_{\pp}(x) = 0$ for all $\pp \in F$. Then there exist elements $a,b \in B$ with $v_{\pp}(a)=v_{\pp}(b)=0$ for all $\pp \in F$ and $x = \frac{a}{b}$. 
\label{shorten}
\end{cor}
\begin{proof} $v_{\pp}(x) = 0$ and $x=\frac{a}{b}$ with $a\notin \pp$ implies $b \notin \pp$.
\end{proof}
 
 \begin{theorem} Let $B$ be an internal Dedekind ring with quotient field $L$. Suppose $\mm$ is a maximal ideal of $B$ and $\mathcal{F}$ is the corresponding ultrafilter. Let $G$ be the value group of the valuation ring $B_{\mm}$. Then
 \[ G \cong \left( {^*\bigoplus}_{\mathcal{M}} {^*\zz} \right) / \ \mathcal{F} \ . \]
Two internal Weil divisors are identified if they coincide on an index set which is contained in $\mathcal{F}$. If $\mm$ is internal, then $\mathcal{F}$ is principal and generated by $\mm$; in this case $G \cong \nsz$.
\label{value-group}
 \end{theorem}   
  \begin{proof} $G$ is isomorphic to the ordered group of {\em principal} fractional ideals of $B_{\mm}$.  
Since $B_{\mm}$ is a valuation ring, this is the same as the group $J(B_{\mm})$ of {\em all} internal fractional ideals of $B_{\mm}$.
We denote the group of internal fractional ideals of $B$ by $J(B)$. There is a natural projection $\pi: J(B) \twoheadrightarrow G$ and the kernel consists of all
fractional ideals $\aaa \in J(B)$ with the property $\aaa \cdot B_{\mm} = B_{\mm}$.
Consider the following diagram:
 
\begin{figure*}[h!]
\begin{center}
\begin{tikzpicture}

 \node (k) at (0,0) {${^*\bigoplus}_{\mathcal{M}} {^*\zz}$};
 \node (sk) at (4,0) {$\left( {^*\bigoplus}_{\mathcal{M}} {^*\zz} \right) / \ \mathcal{F} $};
 \node (kx) at (0,-2) {$J(B)$};
 \node (sky) at (4,-2) {$J(B_{\mm}) \cong G$};
 
l  \draw[->>, thick]  (k)  edge node[above] {$\pi_1 $}  (sk);
  \draw[->>, thick]  (kx)  edge node[above] {$\pi$}  (sky);
   \draw[->, thick]  (k)  edge node[right] {$\cong$} node[left] {$\alpha$} (kx);
     \draw[->, thick, dashed]  (sk)  edge node[right] {$\beta$}  (sky);
       \draw[->>, thick]  (k)  edge node[right] {$\ \pi_2$}  (sky);
          
\end{tikzpicture}
\caption {Internal divisors of $B$ (left) and $B_{\mm}$ (right).}
\end{center}
\end{figure*}
The horizontal maps $\pi_1$ and $\pi$ are projections, and the vertical isomorphism $\alpha$ is given by
Proposition \ref{ideals}.  $\alpha$ maps an internal divisor $(x_{\pp}) \in {^*\bigoplus}_{\mathcal{M}} {^*\zz}$ to the internal fractional ideal 
$\aaa = \prod_{\pp \in \mathcal{M}} \pp^{x_{\pp}} \in J(B) $. Internal fractional ideals are finitely generated and the image $\pi(\aaa) \in G$ has a single generator $x \in L$, since $B_{\mm}$ is a valuation ring.

We claim that the dashed vertical map $\beta$ exists and is an isomorphism. Since $\pi_2$ is surjective we only need to show
that $\ker(\pi_1)=\ker(\pi_2)$.

Let $(x_{\pp}) \in \ker(\pi_1)$. Then there exists a hyperfinite set $F \in \mathcal{F}$ such that $x_{\pp}=0$ for all $\pp \in F$.
Let $\aaa = \prod_{\pp} \pp^{x_{\pp}} \in J(B)$ be the associated
fractional ideal of $B$. 
By Corollary \ref{shorten} we have that $\pi(\aaa)=\pi_2((x_{\pp}))$ is generated by $x=\frac{a}{b}$ with $a,b \in B$ and $v_{\pp}(a)=v_{\pp}(b)=0$ for all $\pp \in F$. This implies $V(a) \cap F = V(b) \cap F = \varnothing$.
Obviously, $V(a) \notin \mathcal{F}$, since otherwise $V(a) \cap F = \varnothing \in \mathcal{F}$, a contradiction. 
Similarly, $V(b) \notin \mathcal{F}$.
This shows that $a,b \notin \mm$, so that $x$ is a unit in $B_{\mm}$ which implies $(x_{\pp}) \in \ker(\pi_2)$.

For the inverse inclusion, let $(x_{\pp}) \in \ker(\pi_2)$. Then $\aaa = \prod_{\pp} \pp^{x_{\pp}} \in J(B)$ is the corresponding fractional ideal and
$\pi(\aaa)=\pi_2((x_{\pp}))=x \cdot B_{\mm} \in G$ with $x \in L$.  By assumption, $x$ is a unit in $B_{\mm}$
and hence $x=\frac{a}{b}$ with $a,b \in B \setminus \mm$. Thus $V(a), V(b) \notin \mathcal{F}$. Choose an arbitrary $F \in \mathcal{F}$. Since $\mathcal{F}$ is a prime filter, $F \setminus V(a) \in \mathcal{F}$, $F \setminus V(b) \in \mathcal{F}$
and thus
$\widetilde{F}=(F \setminus V(a)) \cap (F \setminus V(b)) \in \mathcal{F}$. Then $a \notin \pp$ and $b \notin \pp$ for all
$\pp \in \widetilde{F}$ and therefore $v_{\pp}(x)=0$ for all $\pp \in \widetilde{F}$. We conclude that $(x_{\pp}) \in \ker(\pi_1)$.
\end{proof}

The next Proposition shows that $\zz$ is a natural subgroup of $G$.
 
 \begin{prop} Let $G=\left( {^*\bigoplus}_{\mathcal{M}} {^*\zz} \right) / \ \mathcal{F}$ be as in the Theorem above.
 \begin{enumerate}
 \item There is a monomorphism of ordered groups $\zz \hookrightarrow G$; choose any $F \in \mathcal{F}$ and map $k \in \zz$ to $(x_{\pp}) \in {^*\bigoplus}_{\mathcal{M}} {^*\zz} $ where $x_{\pp} = k$ if $k \in F$ and zero otherwise. This map does not depend on the choice of $F \in \mathcal{F}$.
 \item The image of $\zz$ under the monomorphism (a) coincides with the image of ${^*\bigoplus}_{\mathcal{M}\ } {\zz}$ in $G$.
 \item If $\mathcal{F}$ is principal, then the monomorphism (a) is the given by the inclusion $\zz \subset \nsz$.
 \end{enumerate}
 \label{zz}
 \end{prop}
 \begin{proof} Two sequences yield the same element in  $G$, if they coincide on any $F \in \mathcal{F}$. The map is a well defined homomorphism, preserves the order and is injective. This shows (a).\\ 
 Let $(x_{\pp}) \in {^*\bigoplus}_{\mathcal{M}} {^*\zz}$ be in the image of ${^*\bigoplus}_{\mathcal{M}\ } {\zz}$. Then there exists a set $F \in \mathcal{F}$ such that
 $x_{\pp} \in \zz$ for all $\pp \in F$. Since $F$ is hyperfinite and the direct sum is internal, the set $H=\{ x_{\pp}\ |\ \pp \in F \}$ must be hyperfinite, too.
 Since $H$ is also a subset of  $\zz$, we conclude that $H$ must be finite. Let $k \in \zz$ and $F_k \subset F$ the subset of $\pp \in F$ where $x_{\pp}=k$ holds. Only a finite number of sets $F_k$ are non-empty and $F = \dot{\bigcup}_{k \in \zz} F_k$. Hence there exists $k_0 \in \zz$ such that $F_{k_0} \in \mathcal{F}$.
This implies that $(x_{\pp})_{\pp \in \mathcal{M}}$ and the image of $k_0$  give the same element in $G$ which proves (b).\\
(c) follows from Proposition \ref{value-group} and (a).   
\label{subgroups}
 \end{proof}

We use valuation theory in order to classify the ideals of the localization at a maximal ideal. Recall that the prime ideals of $B$ are contained in exactly one maximal ideal.

\begin{prop} Let $\mm$ be a maximal ideal of the internal Dedekind ring $B$ and $B_{\mm}$ the valuation ring having value group $G$. There is a bijection between the
ideals $\aaa$ of $B_{\mm}$ and convex sets $U \subset G$ of non-negative elements.
The prime ideals $\pp \subset \mm$ correspond to convex subgroups $U \subset G$.
The convex subgroups $U \subset G$ and the prime ideals $\pp \subset \mm$ are totally ordered. 
\label{class-prime-ideals}
\end{prop}
Recall that a subset $U \subset G$ is called {\em convex}, if $x,z \in U$ and $x \leq y \leq z$ implies $y \in U$.
The associated ideal $\aaa$ consists of all $a \in B_{\mm}$ with $x < v(a)$ for all $x \in U$ or $a=0$.
A subgroup $U \subset G$ is {\em convex} or {\em isolated}, if $x \in U$ and $0 \leq y \leq x$ implies $y \in U$.
The elements $a \in \pp$ of the corresponding prime ideal satisfy $x < v(a)$ for all $x \in U$ or $a=0$.
\begin{proof}
See \cite{schilling1950} Corollary 1 and Lemma 11.\end{proof}

\begin{remark} 
Let $\mm$ and $G$ be as above. A convex subgroup
$U$  corresponds to the prime ideal $\pp$ with $a \in \pp$ iff $x<v(a)$ for all $x \in U$ or $a=0$.  If $U=\{0\}$ then $\pp = \mm$. Further, $U=\nsz$ yields $\pp=(0)$.  

The smallest nonzero convex subgroup of $G$ is $\zz$ (see Proposition \ref{zz}). The corresponding prime ideal $\pp$ consists of all $a \in \mm$ with the property that $v(a) \notin \zz$. We write $\pp= \mm^{\infty}$ and this is the largest prime ideal which is strictly contained in $\mm$.
If $\mm$ is internal then $G = \nsz$ and $\pp=\mm^{\infty}=\bigcup_{N} \mm^N$, where $N \in \nsn \setminus \nn$.  
It was shown in \cite{klingen1975} 2.13 that $\mm$ contains a decreasing chain of prime ideals of length $\# (\nsn)$.
The chains of prime ideals of $B=\nsz$ for different types of maximal ideals (standard, internal, external) are shown in Figure \ref{chain}.
\hfill$\lozenge$ 
\label{ex-prime-ideals}
\end{remark}

\begin{figure*}[h!]
\begin{center}
\begin{tikzpicture}

\node (i) at (1,4.5) {standard};
\node (i) at (4,4.5) {internal};
\node (i) at (7,4.5) {external};
\node(idn) at (10,4) {$U=\{0\}$};
\node(idz) at (10,2.5) {$U=\zz \subset G$};
\node(idu) at (10,1.25) {$\zz \subsetneq U \subsetneq G$};
\node(idg) at (10,0) {$U=G $};
 \node (p) at (0.85,4) {$(p)$};
 \node (P) at (4,4) {$(P)$};
 \node (m) at (7.2,4) {$\mm$};
 \node (pu) at (2,2.5) {$(p)^{\infty}$};
 \node (Pu) at (4,2.5) {$(P)^{\infty}$};
 \node (mu) at (6,2.5) {$\mm^{\infty}$};
 \node (z) at (4,0) {$(0)$};
 
  \draw[thick]  (p)  edge node[right] {}  (pu);
  \draw[thick]  (P)  edge node[right] {}  (Pu);
    \draw[thick]  (m)  edge node[right] {}  (mu);
  \draw[dotted, thick]  (pu)  edge node[right] {}  (z);
  \draw[dotted, thick]  (Pu)  edge node[right] {}  (z);
    \draw[dotted, thick]  (mu)  edge node[right] {}  (z);

\end{tikzpicture}
\caption {Chains of prime ideals in $\nsz$. Convex subgroups $U \subset G$ correspond to prime ideals.}
\label{chain}
\end{center}
\end{figure*}
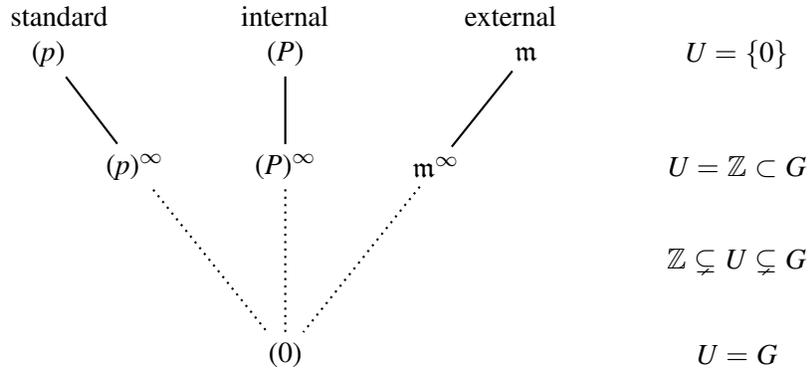

\begin{prop} Let $\mm$ be a maximal ideal of an internal Dedekind ring $B$, $B_{\mm}$ the valuation ring, $L$ the quotient field and $v: L \rightarrow G \cup \{\infty\}$ the corresponding valuation. Suppose that  $U$ is a convex subgroup of $G$ and $\pp$ is the prime ideal corresponding to $U$. This induces a valuation $v_{/U} :\ L \rightarrow (G/U) \cup \{\infty\}$ and the valuation ring  with respect to $v_{/U}$ is 
the localization  $B_{\pp}$. Furthermore, $v$ induces a valuation $v_{|U}$ on the residue field $k_{\pp} := B_{\pp} / \pp$ with value group $U$.
\label{isol-val}
\end{prop}
\begin{proof}  This follows from \cite{schilling1950} Theorem 5. 
\end{proof}

In the next Section we will interpret  $v_{/U}$ as a {\em vertical generization} and $v_{|U}$ as a {\em horizontal specialization} of $v$.

We have seen above that prime ideals yields valuations. Conversely, suppose now that a valuation is given.

\begin{prop} Let $B$ be a Pr\"ufer ring, $L=Quot(B)$ and 
$v$ a non-trivial valuation on ${L}$ with valuation ring $R_v \supset B$ and valuation ideal $\mm_v$. Define $\pp = \mm_v \cap B$. Then $\pp$ is a nonzero prime ideal of $B$ and $R_v = B_{\pp}$. 
\label{val}
\end{prop}
\begin{proof}
Since $\mm_v$ is maximal, it follows from the definition that $\pp$ is a prime ideal. If $\pp$ were the zero ideal, then $v$ would be trivial.
The inclusion $B_{\pp} \subset R_v$ follows from the definition of $\pp$. For the reverse inclusion, let 
$x \in R_v$  and suppose $x \in R_v \setminus B_{\pp}$. 
Since $B$ is a Pr\"ufer ring, $B_{\pp}$ is a valuation ring and hence $x^{-1} \in B_{\pp}$.
This implies $x=\frac{a}{b}$ with $a \in B \setminus \pp$ and $b \in B$, $b\neq 0$. Furthermore, we have $b \in \pp$ since otherwise $x \in  B_{\pp}$. But then $a = xb \in \mm_v \cap B = \pp$, a contradiction. 
\end{proof}

Proposition \ref{val} implies the following result:

\begin{prop}
There is a one-to-one correspondences between the set of prime ideals of a Pr\"ufer ring $B$ and the valuation rings of $L=Quot(B)$ which contain $B$.
\label{bijprimeval}
\end{prop}

\begin{proof} Any prime ideal $\pp$ yields a valuation ring $B_{\pp}$. Conversely, a valuation ring $R_v$ gives the prime ideal
$\pp = \mm_v \cap B$ and $B_{\pp}=R_v$. Obviously, the zero ideal yields the trivial valuation.
\end{proof}

The above Proposition \ref{bijprimeval} holds for  internal Dedekind rings since they are Pr\"ufer rings (see Corollary \ref{pruefer}).
Together with Theorem \ref{value-group} and Proposition \ref{class-prime-ideals} we obtain a description of the prime spectrum.
Topological statements are given below in Propositions  \ref{maxspec}, \ref{topsurj} and \ref{specrz}.
\begin{cor} Let $B$ be an internal Dedekind ring with quotient field $L$. Then there are bijections between the following sets:
\begin{enumerate}
\item Nonzero prime ideals $\pp \subset B$,
\item Nontrivial valuations $v$ on $L$ with 
$v(B) \geq 0$,
\item Pairs $(\mathcal{F},U)$, where $\mathcal{F}$ is a maximal filter on $\mathcal{L}$ and $U$ is a proper convex subgroup of $\left({^*\bigoplus}_{\mathcal{M}}  \nsz \right) / \mathcal{F}$.
\end{enumerate}
\label{specbij}
\end{cor}
We will see below (Proposition \ref{pseudo-arch}) that the condition $v(B) \geq 0$ since there exist valuations on nonstandard number fields where $\nsz$ and $B$ are not contained in the valuation ring.

Let $\mathcal{S}(\mathcal{L})$ be the {\em Stone space} of $\mathcal{L}$, i.e.\ the set of ultrafilters (maximal filters or equivalently prime filters) on $\mathcal{L}$. A subbase of open sets is formed by the sets 
$$\mathcal{D}(a)=\{\mathcal{F}\ |\ \mathcal{F} \text{ is ultrafilter on $\mathcal{L}$ and } V(a) \notin \mathcal{F} \}$$
 where $a \in B$, $a \neq 0$ and $V(a)= \{ \mm \in MaxSpec(B) \ |\ a \in \mm \}$.
Instead of $V(a)$, one could also use elements $F \in \mathcal{L}$, i.e.\ hyperfinite sets of internal maximal ideals, and define a base of  open sets by $\mathcal{D}(F)=\{\mathcal{F}\ |\ \mathcal{F} \text{ is ultrafilter and } F \notin \mathcal{F} \}$.

\begin{prop} Let $MaxSpec(B)$ be the maximum spectrum of $B$ with the Zariski topology and $\mathcal{S}(\mathcal{L})$ the Stone space of $\mathcal{L}$. Then the map 
$$Supp:\ MaxSpec(B) \rightarrow \mathcal{S}(\mathcal{L})$$
 is a homeomorphism.
 \label{maxspec}
\end{prop}
\begin{proof} Bijectivity follows from Corollary \ref{class-max-ideals} (b).\\
 An open set 
$D_{max}(a)$ $= \{ \mm \in MaxSpec(B)\ |\ a \notin \mm \}$ is mapped onto
the open set $\mathcal{D}(a)$. This shows the continuity.
\end{proof}

\begin{remark} Let $\aaa$ be any ideal in $B$. Then the closed set 
$$V(\aaa) = \{ \mm \in MaxSpec(B) \ |\ \aaa \subset \mm \}$$ is mapped onto
the closed set 
$$\{  \mathcal{F}\ |\ \mathcal{F} \text{ is ultrafilter on $\mathcal{L}$ and } V(a) \in \mathcal{F} \text{ for all } a \in \aaa \ \text{ with } a \neq 0\} $$
of the Stone space $\mathcal{S}(\mathcal{L})$. All closed sets are of this type.
\end{remark}

\begin{prop} Let $f: Spec(B) \setminus \{(0)\} \rightarrow MaxSpec(B) \cong \mathcal{S}(\mathcal{L})$ map a nonzero prime $\pp$ ideal to the unique maximal ideal $\mm \supset \pp$. Then $f$ is surjective and continuous with respect to the Zariski topology and the inclusion $MaxSpec(R) \subset Spec(B) \setminus \{(0) \}$ is a continuous section. The fiber $f^{-1}(\mm)$ is  homeomorphic to the ordered set of convex subgroups (with respect to the right order topology) of the value group $G$ of $B_{\mm}$.
\label{topsurj}
\end{prop} 
\begin{proof} $f$ is continuous since for any $a \in B$, $a \neq 0$, the open set 
$$D(a)=\{ \pp \in Spec(B)\setminus \{(0)\} \ |\ a \notin \pp \}$$
 is mapped onto 
the open set $\mathcal{D}(a)$. It follows from Corollary \ref{specbij} that the fiber $f^{-1}(\mm)$ is in bijection with the set of convex subgroups of $G$. 
Now $f^{-1}(\mm) \cap D(a)$ is the set of nonzero prime ideals $\pp \subset \mm$ with $a \notin \pp$. Let $U \subset G$ be the subgroup that corresponds to $\pp$. Then $a \notin \pp$ iff $v(a) \in U$. The set of subgroups $U$ that satisfy the latter condition is open in the right order topology. Furthermore, the open sets have a subbase of this type. This shows the homeomorphism.
\end{proof}

Finally, we want to relate the spectrum $Spec(B)$ and the {\em internal spectrum} ${^*Spec}(B)$. The latter space contains only the {\em internal prime ideals} of $B$ and is a subset of $Spec(B)$. By transfer, internal nonzero prime ideals are maximal and thus ${^*Spec}(B) \setminus \{(0) \} = {^* MaxSpec(B)} = \mathcal{M}$. 

${^*Spec}(B)$ is a $^*$topological space (only internal unions are permitted) and  the topology for which the internal open sets is a base is called the {\em $Q$-topology} \cite{robinson1996}.  Since the $Q$-topology on ${^*Spec}(B)$ is generated by the open sets $D(a)$, $a \in B$, we obtain the following statement:

\begin{prop} ${^*Spec}(B)$ is a subset of $Spec(B)$ and the Q-topology on ${^*Spec}(B)$ coincides with the subspace topology.
\end{prop}

\section{Riemann-Zariski Space and Valuation Spectrum}
We want to describe the Riemann-Zariski space and the valuation spectrum of an internal Dedekind ring $B$ and its quotient field $L$. Recall that the 
{\em valuation spectrum} $Spv(R)$ of a ring $R$ is the set of equivalence classes of valuations on $R$ with the topology generated by the subsets
$$ D_{x/y} = \{ v \in Spv(R)\ |\ v(x) \geq v(y) \text{ and } v(y) \neq \infty \} $$
where $x,y \in R$ (see \cite{huberknebusch}, \cite{wedhorn2012}). The valuation spectrum of a field $F$ is also called {\em Riemann-Zariski space} $RZ(F)$. 
The topology is generated by the basis
$$ D_x = \{  v \in RZ(F) \ |\ v(x) \geq 0 \} $$
where $x \in F$. Hence $D_x$ is the set of valuations $v$ such that $x$ lies in the valuation ring of $v$. 

 In the previous Section we described the set of  valuations $v \in RZ(L)$, where the valuation ring of $v$ contains $B$ (Proposition \ref{bijprimeval}). We say that $v$ is {\em centered} in $B$. This defines the 
  {\em Riemann Zariski} space $RZ(L,B)$ of $L$ relative to $B$ and we endow it with the subspace topology.   Therefore
$$ RZ(L,B) = \bigcap_{x \in B} D_x $$
is an intersection of open sets in $RZ(L)$.

One says that $v$ is a {\em specialization } of $w$, or that $w$ is a {\em generization} of $v$, if $w \in \overline{\{v\}}$. For valuations of a field $L$, 
it is well known that $v$ is a specialization of $w$ iff the valuation rings satisfy $R_v \subset R_w$ (see \cite{wedhorn2012} 4.11).

\begin{prop} The following map is a homeomorphism:

$$ Spec(B) \stackrel{\cong}{\longrightarrow} RZ(L,B) \ ,\ \pp \mapsto B_{\pp} $$
\label{specrz}
\end{prop}

\begin{proof} Bijectivity was shown in Proposition \ref{bijprimeval}. For $x \in B$ and $x\neq 0$, the Zariski-open set 
$D(x) = \{ \pp \in Spec(B)\ |\ x \notin \pp\}$ is mapped onto to the Riemann-Zariski open set $D_{1/x} \cap RZ(L,B)$. This gives the homeomorphism.
\end{proof}

\begin{cor} 
\begin{enumerate}
\item A valuation $v$ is a closed point in $RZ(L,B)$ iff the corresponding prime ideal in $B$ is maximal. 
\item Let $v,w$ be valuations in $RZ(L,B)$ and $\pp_v, \pp_w$ the corresponding prime ideals in $B$. Then $v$ is a specialization of $w$ iff $\pp_w \subset \pp_v$.
\end{enumerate}
\end{cor}

The following Proposition is a general fact for valuation rings (see \cite{wedhorn2012} 4.12).

\begin{prop} Let $v$ be a valuation on $L$ with value group $G$.
Assume that $U$ is a convex subgroup of $G$ and $w=v_{/U}$ is the corresponding valuation of $L$ having value group $G/U$. Then $w$ is a generization of $v$ and this construction gives all generizations of $v$. 
\label{general}
\end{prop} 

\begin{example}
If $v \in RZ(L,B)$ is a closed point so that the value group is $G=\left({^*\bigoplus}_{\mathcal{M}}  \nsz \right) / \mathcal{F}$, then the generizations of $v$ correspond to convex subgroups of $G$.
The smallest nonzero convex subgroup is $\zz$.
\hfill$\lozenge$

\end{example}

Next, we show that $RZ(L)$ is strictly larger than $RZ(L,B)$. The {\em internal} valuations of $L$ correspond to internal maximal ideals of $B$ and are therefore centered in $B$. But there exist {\em external} valuations on $L$ that are not centered in $B$. Such valuations can be defined using Archimedean absolute values. The construction is similar to a generization, but the underlying absolute values is not a valuation.

\begin{prop} Let $L$ be an internal number field and suppose $|\ | : L \rightarrow \nsr$ is an internal Archimedean absolute value. Let $U$ be a convex subgroup of $\nsr$ which contains the finite hyperreal numbers $\fin(\nsr)$. Define  the map $v: L^{\times} \rightarrow \nsr/U$ by setting $v(x) \equiv - \log|x| \mod U$. Then $v$ is a valuation. 
\label{pseudo-arch}
\end{prop}

\begin{proof} Obviously, $v$ is a homomorphism with $v(1)=0$. An Archimedean absolute value does not satisfy the strict triangle inequality, but
$$ |x+y| \leq 2 \max\{|x|,|y|\}$$
holds and hence $$ v(x+y) \geq -\log(2) + \min\{v(x),v(y)\} $$
 (compare  \cite{robinson-roquette}). Since $-\log(2) \in U$, we conclude that $v$ is a valuation.
\end{proof}

\begin{remark}
By abuse of notation, we call such valuations {\em pseudo-Archimedean}. 
For all nontrivial proper convex subgroups $U$ of $\nsr$, there exist infinite numbers $N \in \nsz$ such that $v(N) \equiv -\log|N| < 0 \mod U$. Thus $\nsz$ is not a subset of the valuation ring and such valuations are not centered in $B$. 
Furthermore, all integers $x \in \zz$ with $x \neq 0$ satisfy $v(x) \equiv 0 \mod U$.
Note that pseudo-Archimedean valuations are external (because $U$ is external), although the underlying nonstandard absolute value (with values in $\nsr$) is internal.  Such valuations (with $U=\fin(\nsr)$) are used in arithmetic applications of nonstandard analysis \cite{robinson-roquette}.\hfill$\lozenge$
\end{remark}

The following Example shows that there exist valuations of $L$ which are neither centered in $B$ nor pseudo-Archimedean. 

\begin{example} Let $L$ be an internal number field with ring of integers $B$. Choose any prime number $x \in \nn$. Then select a non-unit $y \in B$ which is transcendent over $\zz$ and relatively prime to $x$ so that $x$ and $y$ are not contained in any maximal ideal of $B$. For example, choose $y$ to be a nonstandard (i.e.\ infinite) prime number.  
Set $A=\zz[x,y] = \zz[y] \subset L$. Then $(x,y)$ is a maximal ideal in $A$, but $(x,y)=(1)$ in $B$. It follows that there exists a valuation ring $R$ with $A \subset R \subset L$ such that $x$ and $y$ are contained in the valuation ideal of $R$ (see \cite{matsumura} 10.2). The valuation is not centered in $B$ and also not pseudo-Archimedean since the integer $x$ lies in the valuation ideal.\hfill$\lozenge$
\label{ex-notrz}
\end{example}

Now we consider valuations of $B$. 
There is a natural inclusion of valuation spectra $Spv(L) \subset Spv(B)$ as shown in Figure \ref{spvspec1}.
The support $supp(v)$ of a valuation $v: B \rightarrow G \cup \{\infty\}$ is defined by $v^{-1}(\{\infty\})$  and gives  a prime ideal in $B$.

\begin{figure*}[h!]
\begin{center}
\begin{tikzpicture}

 \node (rz) at (0,0) {$RZ(L,B)$};
 \node (spvL) at (3,0) {$Spv(L)$};
 \node (spvB) at (7,0) {$Spv(B)$};
 \node (specL) at (3,-2) {$Spec(L)$};
 \node (specB) at (7,-2) {$Spec(B)$};
 
  \draw[->, thick]  (rz)  edge node[above] {$\subset$}  (spvL);
  \draw[->, thick]  (spvL)  edge node[right] {}  (specL);
   \draw[->, thick]  (spvB)  edge node[right] {supp} node[left] {} (specB);
     \draw[->, thick]  (spvL)  edge node[above] {$\subset$}  (spvB);
       \draw[->, thick]  (specL)  edge node[above] {$(0) \mapsto (0)$}  (specB);
  
\end{tikzpicture}
\caption {Riemann-Zariski space, valuation spectrum and Zarisiki spectrum of $L$ and $B$.}
\label{spvspec1}
\end{center}
\end{figure*}

A valuation $v: B \rightarrow G \cup \{\infty\}$ with $supp(v)=\pp$ uniquely corresponds to a valuation $\tilde{v}$ on the residue field $k_{\pp}=Quot(B/\pp)$. 
The fiber of the support map over a point $\pp \in Spec(B)$ is hence the valuation spectrum $Spv(k_{\pp})$. 
The commutative diagram in Figure \ref{spvspec2} generalizes the right-hand side of Figure \ref{spvspec1}.

\begin{figure*}[h!]
\begin{center}
\begin{tikzpicture}

  \node (spvL) at (3,0) {$Spv(k_{\pp})$};
 \node (spvB) at (7,0) {$Spv(B)$};
 \node (specL) at (3,-2) {$Spec(k_{\pp})$};
 \node (specB) at (7,-2) {$Spec(B)$};
 
  \draw[->, thick]  (spvL)  edge node[right] {}  (specL);
   \draw[->, thick]  (spvB)  edge node[right] {supp} node[left] {} (specB);
     \draw[->, thick]  (spvL)  edge node[above] {$\subset$}  (spvB);
       \draw[->, thick]  (specL)  edge node[above] {$(0) \mapsto \pp$}  (specB);
  
\end{tikzpicture}
\caption {Valuation spectrum and Zarisiki spectrum of $B$ and $k_{\pp}=Quot(B/\pp)$.} 
\label{spvspec2}
\end{center}
\end{figure*}

The valuation spectrum $Spv(B)$ is a large space: the fiber above a prime ideal $\pp$ consists of all valuations of $k_{\pp}$.   
We can construct interesting elements in $Spv(B)$ by {\em horizontal} (primary) specializations and {\em vertical} (secondary) specializations  (see \cite{huberknebusch}, \cite{wedhorn2012}) of valuations in $RZ(L,B)$. 

Suppose that $w \in RZ(L,B)$  corresponds to a prime ideal $\pp$ (see Proposition \ref{specrz}) and $\pp$ is contained in a maximal ideal  $\mm$. Then $B_{\mm} \subset B_{\pp} \subset L$.  Let $G$ be the value group of the valuation $v$ having the valuation ring $B_{\mm}$. The valuation ring $B_{\pp}$ yields a valuation $w=v_{/U}$ with
value group $G/U$ where $U$ is the convex subgroup of $G$ corresponding to $\pp$ (see Proposition \ref{class-prime-ideals}).
$w $ is a {\em vertical} generization of $v$ (see Proposition \ref{general}). Conversely regarded,  $v$ is a vertical specialization of $w$. 

Furthermore, $v \in RZ(L,B)$ admits a {\em horizontal} specialization $v_{|U} : B \rightarrow U \cup \{ \infty \}$ defined by $v_{|U}(x)=v(x)$ if $v(x) \in U$ and $v(x)=\infty$ otherwise. Note that $v_{|U}$ defines a valuation on $B$ (and on $B_{\pp}$, as well as on the residue field $k_{\pp}$), but not on $L$. One has $supp(v_{|U})=\pp$.
 
\begin{example} Let $\mm$ be an internal maximal ideal of $B$ and $v$ the corresponding valuation on $L$ having value group $\nsz$. Define $\pp = \mm^{\infty}$. The associated convex subgroup is $\zz \subset \nsz$. This gives the horizontal specialization $v_{|\zz} : B \rightarrow \zz \cup \{\infty\}$ of $v$ with $supp(v_{|\zz}) = \pp$.
Suppose that $K$ is a number field with ring of integers $\oo_K$ and maximal ideal $\mm$, $L={^*K}$, $B=\noo_K$ and prime ideal $\pp=\mm^{\infty}$.
Then the residue field $k_{\pp}=Quot(B/\pp)$ is isomorphic to the standard completion $K_{\mm}$ of $K$ and the valuation $v_{|\zz}$ is the usual $\mm$-adic valuation.\hfill$\lozenge$ 
\label{horiz-int}
\end{example}  

If $\mm \subset B$ is any (internal or external) maximal ideal and $v \in RZ(L,B)$ the corresponding valuation having value group $G$ (see Theorem \ref{value-group}), then $v$ admits various  vertical generizations and horizontal specializations (see Figure \ref{vert-horiz}) associated to the convex subgroups of $G$. The subgroup $\zz$ of $G$ is particularly interesting. 
Similar to the above Example \ref{horiz-int}, let $\pp = \mm^{\infty}$ be the prime ideal corresponding to $\zz$ and $k_{\pp} = Quot(B/\pp)$ the residue field. Then $v_{/\zz}$ is a generization and $v_{|\zz}$ a specialization of $v$.
The trivial valuations on the residue fields $k_{\pp}$ 
and $k_{\mm}=B/\mm$ are horizontal specializations of $v_{/\zz}$ and $v_{|\zz}$.  The trivial valuation on $L=Quot(B)$ is the generic point of $Spv(B)$.   
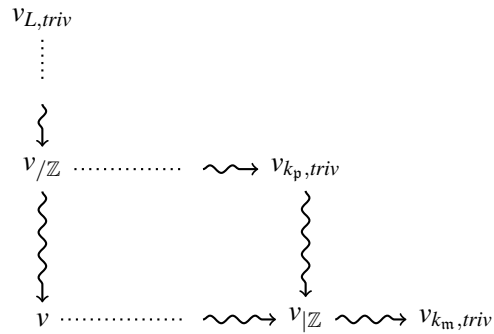
\begin{figure*}[h!]
\begin{center}
\begin{tikzpicture}

 \node (v) at (0,0) {$v$};
 \node (vz) at (0,2) {$v_{/\zz}$};
 \node (v0) at (2,0) {};
 \node (vz0) at (2,2) {};
 \node (vz0v) at (0,3) {};
 \node (vzh) at (3.5,0) {$v_{|\zz}$};
 \node (vth) at (5.5,0) {$v_{k_{\mm},triv}$};
 \node (vtv) at (0,4) {$v_{L,triv}$};
 \node (vat) at (3.5,2) {$v_{k_{\pp},triv}$};

  \draw[->, thick, decorate, decoration={snake,amplitude=.5mm,segment length=3mm,post length=1mm}]  (vz) -- (v);
 \draw[dotted, thick]  (v)  edge node[right] {}  (v0);
  \draw[dotted, thick]  (vtv)  edge node[right] {}  (vz0v);
    \draw[dotted, thick]  (vz)  edge node[right] {}  (vz0);
   \draw[->, thick, decorate, decoration={snake,amplitude=.5mm,segment length=3mm,post length=1mm}]  (v0) -- (vzh);
     \draw[->, thick, decorate, decoration={snake,amplitude=.5mm,segment length=3mm,post length=1mm}]  (vzh) -- (vth);
       \draw[->, thick, decorate, decoration={snake,amplitude=.5mm,segment length=3mm,post length=1mm}]  (vat) -- (vzh);
           \draw[->, thick, decorate, decoration={snake,amplitude=.5mm,segment length=3mm,post length=1mm}]  (vz0) -- (vat);
               \draw[->, thick, decorate, decoration={snake,amplitude=.5mm,segment length=3mm,post length=1mm}]  (vz0v) -- (vz);
  
\end{tikzpicture}
\caption {Vertical generizations and horizontal specialisations of $v \in RZ(L,B)$. The support of the valuations in the columns is $(0)$, $\pp$ and $\mm$, respectively.}
\label{vert-horiz}
\end{center}
\end{figure*}

\begin{remark} In the above Example \ref{horiz-int} we obtained the standard $\mm$-adic valuation in  $Spv(K_{\mm}) $ as a horizontal specialisation of the the $\mm$-adic valuation on ${^*K}$. Note that 
$Spv(K_{\mm})$ can be quite large (see Example \ref{ex-other-val} below), whereas $RZ(K_{\mm},\oo_{K_{\mm}})$ only consists of the $\mm$-adic and the trivial valuation. Similarly, $Spv(B)$ and $Spv(L)$ are much bigger than $RZ(L,B)$ which is homeomorphic to $Spec(B)$ (see Proposition \ref{specrz}).\hfill$\lozenge$
\end{remark}

\begin{example}
Suppose $K=\qq$ and $\mm=(p)$. We construct a valuation on $\qq_p$ which is neither the $p$-adic nor the trivial valuation.
Set $x=p$, $y=1+p+p^2+\dots$ and $A=\zz[x,y]=\zz[y]$. Similar as in Example \ref{ex-notrz},  there exists a valuation ring $R$ with $A \subset R \subset \qq_p$ such that $x$ and $y$ are in the valuation ideal of $R$. Since $(x,y)=1$ in $\zz_p$, the valuation ring $R$ is not equal to $\zz_p$.\hfill$\lozenge$
\label{ex-other-val}
\end{example}

Finally, we describe the valued field $Quot(B/supp(w))$ of valuations $w \in Spv(B)$.

\begin{prop} Let $w \in Spv(B)$ and assume that $\pp=supp(w)$ is nonzero.  Let $\mm$ be the unique maximal ideal above $\pp$.
Then $\mm$ yields a valuation $v$ on $B$ with value group
$G=\left({^*\oplus}_{\mathcal{M}} \nsz \right) / \mathcal{F}$ and $\pp$ corresponds to a convex subgroup $U \subset G$. 
Set $H= G_{\geq 0} \setminus U$. Then $\pp= \mm^H$, where the latter prime ideal is defined by all $x \in \mm$ with $v(x) \in H$ (or $x=0$) and one has
$$ k_{\pp} = Quot(B/\pp) = Quot(B/\mm^H). $$
\end{prop}
\begin{proof} By Proposition \ref{class-prime-ideals} and Remark \ref{ex-prime-ideals}, $x \in \pp$ iff $u<v(x)$ for all $u \in U$ or $x=0$. This is equivalent to $v(x) \in H$. 
\end{proof}

Note that a valued field may have many different valuations (see  Example \ref{ex-other-val}).

\section{Conclusion and Outlook}
We classified the prime ideals and general ideals of nonstandard Dedekind rings $B$. The external ideals are particularly interesting. Prime ideals are contained in exactly one maximal ideal. The set of prime ideals in a given maximal ideal $\mm$ is totally ordered and the largest prime ideal $\pp \subsetneq \mm$ is $\pp=\mm^{\infty}$. If $K$ is a global number field, $B=\noo_K$ and $\mm$ a standard prime ideal, then the 
residue field of $\mm^{\infty}$ gives the standard completion $K_{\mm}$. This can be generalized to the standard adeles and ideles which are  subquotients of ${^*K}$. We will treat this in a separate paper.

The spectrum of a nonstandard Dedekind ring could be described using lattice and valuation theory. The
Riemann-Zariski space and the valuation spectrum are large spaces which we could only partly describe.   Some elements are given by vertical generizations or horizontal specialisations of 
valuations corresponding to maximal ideals in $B$, but other possibilities exist.

The Archimedean places of a number field can be turned into a valuation which is not possible in the standard case. This can be used in arithmetic applications of nonstandard analysis, where standard function fields of varieties are embedded into nonstandard number fields \cite{robinson-roquette}, \cite{kani1980}.

In future work, we hope to apply these results to class field theory, zeta functions and arithmetic geometry.\\

\bibliographystyle{jloganal}

\bibliography{nsabib}

\end{document}